\documentclass[reqno]{amsart}
\usepackage{amsmath, mathrsfs, amssymb}
\allowdisplaybreaks

\newtheorem{theorem}{Theorem}[section]
\newtheorem{lemma}[theorem]{Lemma}
\newtheorem{proposition}[theorem]{Proposition}
\theoremstyle{definition}
\newtheorem{definition}[theorem]{Definition}
\newtheorem{example}[theorem]{Example}
\newtheorem{corollary}[theorem]{Corollary}
\theoremstyle{remark}
\newtheorem{remark}[theorem]{Remark}

\begin{document}

\title{ Functions and Means of Accretive Operators} 

\author{Mitsuru Uchiyama}
\address{Shimane Univ. Matsue, and Ritsumeikan Univ. Otsu, Japan} 
\email{uchiyama@riko.shimane-u.ac.jp}
\address{Shimane Univ. Matsue, and Ritsumeikan Univ. Otsu, Japan} 
\subjclass[2020] {Primary 47A60; 47A64; 15A18; Secondary 15A15}
\date{6/29, 2026}


\commby{Javad Mashreghi}

\begin{abstract}
Let $A$ be a bounded accretive operator on a Hilbert space and 
$f(t)$ an operator monotone function on $(0, \infty)$ with $f(0)>-\infty$. 
Then, for $\epsilon >0$, analytic function $f (A+\epsilon I) $ is defined by Riesz-Dunford integral. 
We define $f(A)$ as the norm limit of it and show    
$$ f(A) = f(0)I + b A + \int_0^{\infty} (\frac{1}{\lambda} I - (\lambda I + A)^{-1})
d\mu(\lambda).$$
This is a generalization of fractional powers 
$$A^r = \frac{\sin r \pi}{\pi} \int_0^{\infty} 
 (\frac{1}{\lambda} I - (\lambda I + A)^{-1}) \lambda ^{r} d\lambda \quad (0<r<1).$$
Let $A$ and $B$ be strictly accretive matrices, namely those real parts are positive definite.
 The geometric mean 
$A\# B$ has been introduced in \cite{Drury} and subsequently general matrix mean 
$A\sigma_f B$ in \cite{BKS1}. 
We extend these means to accretive, not necessarily strictly accretive, operators $A$ 
and $B$, and verify that   
$$A\# B= A^{1/2} B^{1/2}$$
 if $A$ and $B$ are normal and commutative. 
Let $A$ be a strictly accretive operator. Then we show that 
$$0 \leqq \frac{1}{2} (A + A^*) \leqq  A \# A^* \leqq 2(A^{-1} + (A^*)^{-1})^{-1},$$
and that $A \# A^* = | A |$ if and only if $A$ is normal. 
For a normal and strictly accretive operator $A$ we get
\begin{align*}
&|A|=  \frac{1}{\pi}\int_0^{\infty}A (\lambda A + A^*)^{-1} A^* 
\lambda^{-1/2} d \lambda, \\
&A + A^* \leqq   A^{1-r} A^{*r} + A^r A^{*(1-r)} \quad (0\leqq r \leqq 1).
\end{align*}
The former deduces the simplest case  
\begin{center}
$1 = \frac{1}{\pi} \int_0^{\infty} \frac{1}{(\lambda e^{i\theta}  +  e^{-i\theta})
\sqrt{\lambda} } d \lambda \quad (|\theta | < \pi/2), $
\end{center}
which is a basic formula shown by using residue.

\end{abstract}

\maketitle
\section{Introduction}
In this note, an {\it operator} means a bounded linear operator on a Hilbert space with the inner product $(\!\mathbf{x}, \mathbf{y}\!)$. The numerical range 
$\mathcal{W}(T)\!:=\{(\!T\mathbf{x}, \mathbf{x}\!): ||\mathbf{x}|| = 1\}$ 
is convex, and it's closure contains the spectrum $\it{Sp}(T)$.
 $T$ is said to be positive and denoted by $T\geq 0$ 
if $(\!T\mathbf{x}, \mathbf{x}\!)\geq 0$ for every $\mathbf{x}$. 
$A$ is called an {\it accretive} operator if 
 $\Re{A} \geqq 0$, where 
 $A = \Re(A) + i \Im(A)$ is the Cartesian decomposition of $A$. 
 $A$ is said to be {\it strictly accretive} when  
$\Re(A)$ is not only positive but also positive definite (\cite{Kato-book}). 
\par
A real valued function $f(t)$ is called 
an {\it operator monotone} function on the interval $(0, \infty)$ if it satisfies $f(A) \leqq f(B)$ whenever 
$0< \it{Sp}(A)$ and $0\leqq A \leqq B$.
 By the Loewner theorem it has the analytic extension $f(z)$ to $\mathbb{C} \smallsetminus (-\infty, 0]$, 
which is a Pick function, that is,  
it maps the upper open half plane into itself. By Herglotz and Nevanlinna's theorem it is represented as 
\begin{equation*}
f(t)= a + bt + \int_0^{\infty} (\frac{\lambda}{\lambda^2 +1} - \frac{1}{\lambda +t}) d\mu(\lambda), \quad 
\int_0^{\infty}\frac{1}{\lambda^2 + 1} d\mu (\lambda) <\infty, 
\end{equation*}
where $a$ is a real number and $b\geqq 0$. \par
Suppose $f(0):= f(+0) >-\infty$. Then we have 
$$\lim_{t\to +0}\int_0^1 (\frac{\lambda}{\lambda^2 +1} - \frac{1}{\lambda +t}) d\mu(\lambda) >-\infty;$$ 
because, for $0<t<1$, 
$$|\int_1^{\infty} (\frac{\lambda}{\lambda^2 +1} - \frac{1}{\lambda +t}) d\mu(\lambda)|\leqq 
\int_1^{\infty} \frac{\lambda t +1}{(\lambda^2 +1)(\lambda +t)} d\mu(\lambda) \leqq
\int_1^{\infty} \frac{1}{\lambda^2 +1} d\mu(\lambda) < \infty.$$
By Fatou's theorem 
$$\int_0^1\lim_{t\to +0} ( \frac{1}{\lambda +t} - \frac{\lambda}{\lambda^2 +1} )d\mu(\lambda) \leqq 
\lim_{t\to +0}\int_0^1 ( \frac{1}{\lambda +t} - \frac{\lambda}{\lambda^2 +1}) d\mu(\lambda)<\infty.$$
Since $\int_0^1 \frac{\lambda}{\lambda^2 +1} d\mu(\lambda)<\infty$, this implies  
\begin{center}
$\int_0^1 \frac{1}{\lambda} d\mu(\lambda)<\infty$ and hence $\mu(\{0\})=0$.
\end{center}
 We hence get 
\begin{equation}
\begin{split}
&f(t)= f(0) + bt + \int_0^{\infty} (\frac{1}{\lambda} - \frac{1}{\lambda +t}) d\mu(\lambda), \\ 
&b\geqq 0, \quad \int_0^{\infty}\frac{1}{\lambda^2 + 1} d\mu (\lambda) <\infty, \quad \int_0^1 
\frac{1}{\lambda} d\mu(\lambda)<\infty. \label{eq:1}
\end{split}
\end{equation}
For example,  for $0<r<1$ and $0<t<\infty$
$$t^r = \frac{\sin r \pi}{\pi} \int_0^{\infty} \frac{t}{\lambda + t } \lambda ^{r-1} d\lambda =
 \frac{\sin r \pi}{\pi} \int_0^{\infty} 
 (\frac{1}{\lambda} - \frac{1}{\lambda +t}) \lambda ^{r} d\lambda.$$
(refer to \cite{D}, \cite{B-M}, \cite{Simon} for more details)\par
Let $T$ be an operator with ${\it Sp}(T) \cap (-\infty, 0] = \emptyset$ and 
 $\gamma$ a smooth curve 
in $\mathbb{C} \smallsetminus (-\infty, 0]$ and surrounding 
$\it{Sp}(T)$. Then 
the Riesz-Dunford integral
$$f(T) = \frac{1}{2\pi i}\int_{\gamma}  f(z) (z I - T)^{-1} dz$$
 is well-defined for an analytic extension $f(z)$ of $f(t)$ given \eqref{eq:1}
(see Chapter VII of \cite{D-S}).
We will show that
$$ f(T)=  f(0) I  + b T + \int_0^{\infty} \left(\frac{1}{\lambda } I - (\lambda I  + T)^{-1}\right)d\mu(\lambda), $$
and that this holds for an accretive operator $A$ even if $0\in {\it Sp}(A)$.
This naturally deduces 
$$A^r = \frac{\sin r \pi}{\pi} \int_0^{\infty} 
 (\frac{1}{\lambda} I - (\lambda I + A)^{-1}) \lambda ^{r} d\lambda \quad (0<r<1),$$
which is known as the Balakrishnan formula\cite{Bal} for a closed and accretive operator
 (cf. \cite{N-F}, \cite{Kato-book}).
  So, the formula $f(T)$ mentioned above is a generalization of fractional powers. \par
 Drury \cite{Drury}
defined {\it geometric mean} $A\# B$ of strictly accretive {\it matrices} $A$ and $B$ by 
\begin{equation}\label{eq:1-0}
(A\# B)^{-1}:= \frac{2}{\pi} \int_0^{\infty} (tA + t^{-1} B)^{-1} t^{-1} d t.
\end{equation}
and showed $A \# B = A^{1/2} (A^{-1/2} B A^{-1/2})^{1/2} A^{1/2}$. We extend this as follows: 
Let $A$ and $B$ be accretive operators such that  
${\it Sp}(BA^{-1}) \cap (-\infty, 0] = \emptyset$.  
Then we define the geometric mean and the general operator mean $\sigma_f$ by 
$$A\#B:=A^{1/2} (A^{-1/2} B A^{-1/2})^{1/2} A^{1/2}, \quad A\sigma_f B:= A^{1/2} 
f( A^{-1/2} B A^{-1/2}) A^{1/2},$$
and show 
\begin{equation}\label{eq:1-1}
A\#B= \frac{1}{\pi}\int_0^{\infty} \frac{1}{\lambda} 
A ( \lambda A + B)^{-1} B \sqrt{\lambda} d\lambda.\\
\end{equation}
A new formula  
$$A (A\# B)^{-1} B= A \# B$$
leads us to see that  \eqref{eq:1-1}  is coincident with \eqref{eq:1-0}. 
One of the objective of this paper is to show 
$$A\#B = A^{1/2} B^{1/2}$$
if $A$ and  $B$ are normal and commutative. Another one is to verify  
$$0\leqq A^*\nabla A \leqq A^* \# A \leqq A^* \,!\, A,$$
where symbols $\nabla, !$ conventionally stand for arithmetic and harmonic mean, 
respectively. At first sight this seems strange, because 
\begin{center}
$A \nabla B \geqq A \# B \geqq A \,!\, B\geqq 0$ for $A, \,B \geqq 0$.
\end{center}
However we should remember that $0\leqq {\it Re}(\alpha) \leqq |\alpha| \leqq 
|\alpha |^2 /{\it Re}(\alpha)$ for a complex number $\alpha$ with 
${\it Re}(\alpha) >0$.
We further show that $A^* \# A = | A |$ if and only if $A$ is normal. Consequently, for 
a stricly accretive and normal operator $A$  
$$|A|=  \frac{1}{\pi}\int_0^{\infty}A (\lambda A + A^*)^{-1} A^* 
\lambda^{-1/2} d \lambda , $$
which deduces the simplest case:
$$1 = \frac{1}{\pi} \int_0^{\infty} \frac{1}{(\lambda e^{i\theta}  +  e^{-i\theta})
\sqrt{\lambda} } d \lambda \quad (|\theta | < \pi/2). $$
This is a basic formula shown by using residue.
For an accretive and normal operator $A$ we also gain 
$$A + A^* \leqq   A^{1-r} A^{*r} + A^r A^{*(1-r)} \quad (0\leqq r \leqq 1).$$

\section{Preliminaries}
Let $f(t)$ be an operator monotone function on $(0, \infty)$ represented by 
$\eqref{eq:1}$. Then by Loewner's theorem 
\begin{equation}
f(z) = f(0) + b z + \int_0^{\infty} (\frac{1}{\lambda} - \frac{1}{\lambda + z}) d\mu(\lambda) \label{eq:2}
\end{equation}
 is the analytic extension of $f(t)$ to $\mathbb{C} \smallsetminus (-\infty, 0]$. 
But it is possible to see, without using Loewner's theorem,  that $f(t)$ is of the 
form $\eqref{eq:1}$ if and only if 
$f(t)$ is operator monotone on $(0, \infty)$ with $f(0)>-\infty$; for instance see \cite{Hansen}. 
Just in case, we give a proof of it.
$\arg z$ always stands for the principal argument of $z\ne 0$, namely 
$-\pi < \arg z \leqq \pi$.
\begin{lemma}\label{Lem:2-1} Let $f(t)$ be an operator monotone function on $(0, \infty)$ given 
by \eqref{eq:1}. Then 
$f(z)$ given by \eqref{eq:2} is the analytic extension of $f(t)$ to 
$\mathbb{C}\setminus (-\infty, 0]$ such that $\overline{f(\overline{z})} = f(z)$.
 Further, 
for any $\theta: \frac{\pi}{2} \leqq \theta < \pi$, 
$f(z)$ is continuous 
on $$\{ z \in \mathbb{C}: z=0\,\, \text{or}\,\, |\arg z| \leqq \theta \}.$$
\end{lemma}
\begin{proof} The integrand in \eqref{eq:2} is continuous on  $\{(\lambda , z ): 0<\lambda <\infty, 
z\in \mathbb{C} \smallsetminus (-\infty, 0] \}$, 
 and integrable with respect to $\mu(\lambda)$; indeed, 
$$\lim_{\lambda \to 0} \lambda |\frac{1}{\lambda} - \frac{1}{\lambda -z}| =1, \quad 
\lim_{\lambda \to \infty} ( 1+ \lambda ^2) |\frac{1}{\lambda} 
- \frac{1}{\lambda -z}| =|z|.$$
To see the analyticity, take $z_0 \in \mathbb{C} \smallsetminus (-\infty, 0] $ and put
$0<\delta < {\it dist}\{ z_0, (-\infty, 0]\}$. For $z$ s.t. $|z - z_0|<\delta$, 
$$\frac{f(z) - f(z_0)}{z- z_0} = b + \int_0^{\infty} (\frac{1}{(\lambda + z_0) (\lambda + z)} d\mu(\lambda).$$
Since 
\begin{center} 
$|\frac{1}{(\lambda + z_0) (\lambda + z)}| \leqq \frac{1}{|\lambda + z_0| (|\lambda + z_0| - \delta )} $
\end{center} 
and there is $M>0$ such that 
$\frac{1}{|\lambda + z_0| (|\lambda + z_0| - \delta )} \leqq M \frac{1}{1 + \lambda ^2}$ for $\lambda >0$,
we therefore get 
$$ f'(z)= b + \int_0^{\infty} (\frac{1}{ (\lambda + z)^2} d\mu(\lambda).$$
 We next show the continuity of $f(z)$ at $z=0$.
Let $z\ne 0$, $|\arg z| \leqq \theta$ and $|z|\leqq 1$. Then
\begin{center}
$|\frac{z}{\lambda + z}|\leqq 1\, \,\, (\Re{z} \geqq 0), \quad 
|\frac{z}{\lambda + z}|\leqq |\frac{z}{\Im z}|\leqq \frac{1}{\sin \theta}\,\,
 \, (\Re{z} <0)$.
\end{center}
Consequently  
$$|\frac{1}{\lambda} - \frac{1}{\lambda + z} | \leqq \frac{1}{\lambda \sin \theta} \quad (\lambda > 0).$$
On the other hand, we have
$$|\frac{1}{\lambda} - \frac{1}{\lambda + z} | \leqq \frac{2}{\lambda ^2} \quad 
(\lambda \geqq 2).$$ Combining these inequalities lead us to   
$$|\frac{1}{\lambda} - \frac{1}{\lambda + z} | \leqq g(\lambda)
: =
\begin{cases}
\frac{1}{\lambda} \frac{1}{ |\sin \theta|}  &  (0 <\lambda \leqq 2)\\
\frac{2}{\lambda^2} \frac{1}{ |\sin \theta|}  & ( 2\leqq \lambda <\infty).
\end{cases}
$$
Since $g(\lambda)$ is continuous and integrable,
$$\lim_{z\to 0} f(z) = f(0) + \int_0^{\infty} \lim_{z\to 0} (\frac{1}{\lambda} - 
\frac{1}{\lambda + z})
 d\mu(\lambda) = f(0).$$
\end{proof}

Let $T$ be an operator such that ${\it Sp}(T) \cap (-\infty, 0] = \emptyset$ and 
$f(z)$ an analytic function expressed by \eqref{eq:2}. Then  
$f(T)$ is defined by the Riesz-Dunford integral. 
Recall that ${\it Sp}(f(T)) = f ({\it Sp}(T))$ and $f(g(T))= (f\circ g)(T)$, so 
${\it Sp}(T^{r}) \subseteqq \{z: -r\pi <\arg z <r \pi\} $,  and  $(T^{r})^s = T^{rs}$,    
where $0<r, s <1$, and  
$ (T^{1/n})^n = T$, furthermore, we have $(T^{-1})^r = (T^r)^{-1}$ for $0<r<1$.  
We here represent $f(T)$ by using \eqref{eq:2}, that is \eqref{eq:1}.  

\begin{proposition}\label{Prop:2-2}
Let $T$ be an operator with ${\it Sp}(T) \cap (-\infty, 0] = \emptyset$, and $f(t)$ an
 operator monotone function on $(0, \infty)$ represented by \eqref{eq:1}, namely 
\eqref{eq:2}. 
Define $f(T)$ by the Riesz-Dunford integral.   
Then, for all $\mathbf{x}, \mathbf{y}$
$$
 (f(T) \mathbf{x}, \mathbf{y})=  f(0) (\mathbf{x}, \mathbf{y})  + b (T \mathbf{x}, \mathbf{y})
 + \int_0^{\infty} \left(\frac{1}{\lambda}(\mathbf{x}, \mathbf{y}) - 
((\lambda I  + T)^{-1}\mathbf{x}, \mathbf{y})\right)  d\mu(\lambda).$$
 We write this as
\begin{equation}
 f(T)=  f(0) I  + b T + \int_0^{\infty} \left(\frac{1}{\lambda } I - (\lambda I  + T)^{-1}\right) d\mu(\lambda).\label{eq:3}
\end{equation}
In this case we have
$$|| f(T)|| \leqq   |f(0)|  + b ||T|| + \int_0^{\infty} ||\frac{1}{\lambda } I - (\lambda I  + T)^{-1}|| d\mu(\lambda)
<\infty.$$
Moreover,  $f(T)^* = f(T^*)$,  and  $f(T)$ is normal if $T$ is.

\end{proposition}
 
\begin{proof}
Let $\gamma$ be a smooth, simple and closed curve in $\mathbb{C} \smallsetminus (-\infty, 0]$ 
 surrounding $\it{Sp}(T)$. Then 
$$f(T) = \frac{1}{2\pi i} \int_{\gamma}  f(z) (z I - T)^{-1} dz.$$
Since $\frac{1}{2\pi i} \int_{\gamma}  (f(0) + b z) (z I - T)^{-1} dz = f(0) I  + b T$, we need to show 
\begin{gather*}
\frac{1}{2\pi i} \int_{\gamma} \left(\int_0^{\infty} (\frac{1}{\lambda } - \frac{1}{\lambda + z}) d\mu(\lambda) 
((z I - T)^{-1} \mathbf{x}, \mathbf{y})\right) dz\\
= \int_0^{\infty} ((\frac{1}{\lambda } I - (\lambda I  + T)^{-1}) \mathbf{x}, \mathbf{y})  d\mu(\lambda).
\end{gather*}
By the preceding lemma, the left side exists. To use the Fubini theorem
 represent $\gamma$ as 
$ z=\phi (s), \alpha \leqq s \leqq \beta$. Then the left side is equal to 
$$\frac{1}{2\pi i} \int_{\alpha}^{\beta} \left(\int_0^{\infty} (\frac{1}{\lambda} - \frac{1}{\lambda + \phi (s)})
d\mu(\lambda) 
((\phi (s) I - T)^{-1} \mathbf{x}, \mathbf{y}) \right) \phi'(s) d s,$$
which is finite. 
The integrand of the double integral 
$$(\frac{1}{\lambda} - \frac{1}{\lambda + \phi (s)}) ((\phi (s) I - T)^{-1} \mathbf{x}, \mathbf{y}) \phi'(s)$$
is continuous with respect to  $\lambda$ and $s$. Hence the above successive integral is equal to 
\begin{align*}
&\int_0^{\infty}d\mu(\lambda) 
\frac{1}{2\pi i} \int_{\alpha}^{\beta} (\frac{1}{\lambda} - \frac{1}{\lambda + \phi (s)}) ((\phi (s) I - T)^{-1}
 \mathbf{x}, \mathbf{y})  \phi'(s) d s\\
= & \int_0^{\infty}d\mu(\lambda) \frac{1}{2\pi i} \int_{\gamma} (\frac{1}{\lambda} 
- \frac{1}{\lambda + z})
 ((z I - T)^{-1} \mathbf{x}, \mathbf{y})  d z\\
=&  \int_0^{\infty}((\frac{1}{\lambda}I -
(\lambda I + T)^{-1}) \mathbf{x}, \mathbf{y}) d \mu(\lambda). 
\end{align*}
We thus got the required formula.  
Since 
$$\lim_{\lambda \to + 0} \lambda ||\frac{1}{\lambda} I - (\lambda I  + T)^{-1}|| = 1\, \, \text{and} \, \,
\lim_{\lambda \to + \infty} \lambda^2 ||\frac{1}{\lambda} I - (\lambda I  + T)^{-1}|| = || T||,$$ 
 $$\int_0^{\infty} ||\frac{1}{\lambda} I - (\lambda I  + T)^{-1}|| d\mu(\lambda)
<\infty.$$
From \eqref{eq:3} it follows that $f(T)^* = f(T^*)$, and this lead us to the last statement.
\end{proof}

\begin{remark}\label{Re:2-1}\rm
One may think the integral in $\eqref{eq:3}$ as the Bochner integral of the operator valued 
function. Indeed, the integrand is continuous, and 
$$\int_0^{\infty} ||\frac{1}{\lambda} I - (\lambda I  + T)^{-1}|| d\mu(\lambda)
<\infty.$$
Consequently the integral is well-defined as the Bochner integral and 
\begin{align*}
&\left(\int_0^{\infty} (\frac{1}{\lambda} I - (\lambda I  + T)^{-1}) d\mu(\lambda)
\mathbf{x}, \mathbf{y}\right)\\
 =& \int_0^{\infty} \left(\frac{1}{\lambda}(\mathbf{x}, \mathbf{y}) - 
((\lambda I  + T)^{-1}\mathbf{x}, \mathbf{y})\right)  d\mu(\lambda)
\end{align*}
(cf. Section V. 5 of \cite{Y}). 
But it is enough for us to consider $\eqref{eq:3}$ in the weak sense (see 3.26 Definition of \cite{Rudin}). 
This kind of representation was dealt in \cite{Sano}
\end{remark}
\begin{remark}\label{Re:2-2}\rm
(i). Let $T$ be an $n\times n$ normal matrix such that ${\it Sp}(T) \cap (-\infty, 0] = \emptyset$ and $\alpha_1 P_1 + \cdots + \alpha_k P_k  
\, (1\leqq k \leqq n)$ its spectral decomposition. Then  
$$f(T)= \int_{\gamma} f(z) (z-T)^{-1} dz =  f(\alpha_1) P_1 + \cdots + 
f(\alpha_k) P_k.$$
(ii). Let $T$ be a normal operator such that ${\it Sp}(T) \cap (-\infty, 0] = \emptyset$ and 
$\int_{{\it Sp}(T)} z d E_z$ its spectral decomposition. Then the Fubini 
theorem entails that the Riesz-Dunford integral $f(T)$ coincides with the functional calculus 
$\int_{{\it Sp}(T)} f(z) d E_z$.\\
(iii). 
Let $T= \int_{{\it Sp}(T)} z dE_z$ be a normal operator with ${\it Sp}(T) \cap (-\infty, 0) = \emptyset$. Then for $\epsilon>0$, the Riesz-Dunford integral $f(T+ \epsilon I)$ is well-defined. 
By taking $\theta$ so that  
${\it Sp}(T) \subseteq \{ z \in \mathbb{C}: z=0\,\, \text{or}\,\, |\arg z| \leqq \theta \}$, we obtain 
$$||f(T+ \epsilon I) - \int_{{\it Sp}(T)} f(z) dE_z|| \leqq 
 \sup \{|f(z+ \epsilon) - f(z)|: z \in {\it Sp}(T)\} \to 0 \,\, (\epsilon \to 0).$$
\end{remark}

\begin{example}\label{Ex:2-1}
Let ${\it Sp}(T) \cap (-\infty, 0] = \emptyset$. Then 
\begin{equation}
T^r= \frac{\sin r \pi}{\pi} \int_0^{\infty} 
 \frac{1}{\lambda} (\lambda I + T)^{-1}T\lambda ^r d\lambda \quad (0<r<1). \label{eq:4}
\end{equation}
$(T^{1/2})^2 = T$, but, needles to say, there are plenty of $X$ such that $X^2=T$. 
\end{example}
\section{Accretive Operators}
Let $A$ be an accretive operator. Then 
$${\it Sp}(A)\subseteq \{z: \Re(z)\geqq 0\}, \quad  
||(A+ \lambda I)\mathbf{x}||\geq \lambda ||\mathbf{x}|| \,\,\,\, (\forall\lambda > 0, \,\, \forall\mathbf{x}).$$
 This implies that
\begin{equation}
 ||(A + \lambda I)^{-1}||\leqq \frac{1}{\lambda}. 
\label{eq:5}
\end{equation} and that   
$$\frac{1}{\lambda}I - (A + \lambda I)^{-1}$$
 is accretive too. \par
It is obvious that 
$$( A {\bf x}, {\bf x}) = 0\, \Longrightarrow \, \Re(A){\bf x} = 0\,  \Longrightarrow \,A{\bf x}= - A^*{\bf x},
$$ and  
$$\mathcal{N}(A + i \mu I) = \mathcal{N}( A^* - i\mu I)$$
for every real number $\mu$, where  
$\mathcal{N}(T):= \{\mathbf{x}: T \mathbf{x} = 0\}$. 
\begin{lemma}\label{Lem:3-1}
Let $A$ be an accretive operator and $n$ a natural number. Then 
$$\mathcal{N}(A^n) = \mathcal{N}(A).$$
\end{lemma}
\begin{proof}  $\mathcal{N}(A) \subseteqq \mathcal{N}(A^n)$ is apparent. 
For any ${\bf x} \in \mathcal{N}(A^n)$ we have $A A^{n-1} {\bf x} = {\bf 0}$. We hence 
get $A^* A A^{n-2} {\bf x} = A^* A^{n-1} {\bf x} = {\bf 0}$, which 
$A A^{n-2} {\bf x} = {\bf 0}$ follows from. By repeating this way, we arrive at  
$A {\bf x} = {\bf 0}$.
\end{proof}
Let us define a function $f(A)$ of an accretive operator $A$ by $f(z)$ given by \eqref{eq:2}. We need to 
pay attention to that the case $0\in \it{Sp}(A)$ 
is not excluded. 

\begin{theorem}\label{Th:3-1}
Let $A$ be an accretive operator and  $f(t)$ an operator monotone function
 on $(0, \infty)$ expressed by \eqref{eq:1}. 
Then   
\begin{equation}
 f(A): = f(0)I + b A + \int_0^{\infty} (\frac{1}{\lambda} I - (\lambda I + A)^{-1})
d\mu(\lambda) \label{eq:6}
\end{equation}
 is a bounded operator. \\
Let $f(A+ sI)$ be functions defined by the Riesz-Dunford integral for $s>0$. Then 
$$f(A+ sI)  \Longrightarrow f(A) \quad (s \to +0).$$
Further, $f(A)^* = f(A^*)$, and $f(A)$ remains accretive if $f(0)\geqq 0$. 
\end{theorem}
\begin{proof}
 Since \\
$||\frac{1}{\lambda} I - (\lambda I  + A)^{-1}|| \leqq  2\frac{1}{\lambda} $ and $\lim_{\lambda \to + \infty} \lambda^2 ||\frac{1}{\lambda} I - (\lambda I  + A)^{-1}|| = || A||$, 
 $$\int_0^{\infty} ||\frac{1}{\lambda} I - (\lambda I  + A)^{-1}|| d\mu(\lambda)
<\infty.$$
$f(A)$, the right side of \eqref{eq:6}, is therefore well-defined and bounded.  
Express $f(A + sI)$ as \eqref{eq:3}. Then 
\begin{align*}
&( f(A + s I) {\bf x}, {\bf y}) -( f(A) {\bf x}, {\bf y}) = 
 b  s ({\bf x}, {\bf y}) \\ 
&+ \int_0^{\infty} \left(( (\lambda  I + A)^{-1} {\bf x}, {\bf y}) - 
( ((\lambda + s) I + A)^{-1} {\bf x}, {\bf y}) \right)
 d\mu(\lambda),
\end{align*}
and consequently  
$$|| f(A + sI) - f(A)|| \leqq b  s +  \int_0^{\infty} ||(\lambda  I + A)^{-1}   - 
((\lambda + s) I + A)^{-1} || d\mu (\lambda).
$$
Write the integrand $h_s (\lambda)$. Then 
$$ h_s(\lambda) \leqq  g(\lambda) \quad  (0<s<1, \, 0<\lambda <\infty),$$
where 
$$ g(\lambda): =
\begin{cases}
\frac{2}{\lambda}  & \text{on} \,\, 0 <\lambda \leqq \frac{1}{2}\\
\frac{1}{\lambda^2}  & \text{on} \,\, \frac{1}{2}\leqq \lambda <\infty.
\end{cases}
$$
$g(\lambda)$ is obviously integrable. 
Since $\mu(\{0\}) =0$ and 
\begin{center}
$||(\lambda  I + A)^{-1}   - 
((\lambda + s) I + A)^{-1}|| \to 0$ as $s\to + 0$ for $\lambda >0$,
\end{center}
by the Lebesgue theorem,  
$$|| f(A + sI) - f(A)|| \to 0 \quad (s\to +0).$$
 From \eqref{eq:6} it directly follows that $f(A^{*})= f(A)^{*}$ and that 
$f(A)$ is accretive, because $f(0)\geqq 0$ and 
$\frac{1}{\lambda}I - (A + \lambda I)^{-1}$ is accretive.
\end{proof}
This theorem shows that \eqref{eq:4} holds for an accretive operator $A$ 
even if $0\in \it{Sp}(A)$, which is known as the formula by Balakrishnan\cite{Bal} (also see 
\cite{Kato-book, N-F}).\par 
\begin{lemma}\label{Le:2}
Let $\{A_n\}$ be a sequence of accretive operators which converges to an operator 
$A$ in the operator norm. Then for an operator monotone function 
$f(t) >0$ on $(0, \infty)$
$$f (A_n) \Longrightarrow f(A).$$
\end{lemma}
\begin{proof} Since $A$ is accretive, $f(A)$ is given by \eqref{eq:6}. Consequently,  
one can prove this in the same way as the proof of Theorem\ref{Th:3-1}.
\end{proof}
\section{Numerical Range}
\it For $f(t)$ represented by \eqref{eq:1}, 
from now on, we assume $f(t)>0$, namely $f(0)\geqq 0$. 
\rm
\begin{definition}\rm For $0\leqq r \leqq 1$, we write 
$A\in \mathcal{S}_{r}$ if  $|\arg (A {\bf x}, {\bf x}) |\leqq r \pi /2$ unless 
$(A {\bf x}, {\bf x}) = 0$. 
\end{definition}
An accretive operator belongs to $\mathcal{S}_{1}$.

\begin{lemma}\label{Le:4-1}
Let $f(z)$ be its analytic extension to 
$\mathbb{C}\setminus (-\infty, 0]$.   
Then, for $z \in \mathbb{C}$,  
\begin{align*}
 &\Im (z) > 0 \Longrightarrow  0< \arg (f(z)-f(0)) \leqq \arg z < \pi, \\
 &\Im (z) < 0 \Longrightarrow  -\pi < \arg z \leqq \arg (f(z) - f(0)) < 0.
\end{align*}
\end{lemma}
 \begin{proof} We exploit \eqref{eq:2}.
$f(0)\geqq 0$ and $b\geqq 0$.  Suppose $\Im (z) >0$. Then  $\arg b z = \arg z$ and   
 $ 0< \arg(z) -\arg(\lambda + z) < \arg (z) <\pi$ for every $\lambda >0$.  
We hence get the first implication. The second one similarly follows. 
\end{proof}
We now apply one of Kato's mapping theorems.  For a set $X$, $\overline{co}X$ stands for 
the closed convex hull of $X$. 
\begin{proposition}\label{Prop:4-2} 
Let $f(t)>0$ be an operator monotone function on 
$(0, \infty)$. Then, if $A\in \mathcal{S}_{1}$, 
$$\mathcal{W}(f(A)) \subseteqq  \overline{co}\{f(z):  \Re(z) \geqq 0\} \subseteqq \{z: \Re(z) \geqq 0\}.$$
\end{proposition}
\begin{proof} For $s >0$ put 
$f_{s}(z) \!:= f(z + s)$, which is analytic on $\mathbb{C} \smallsetminus (-\infty, -s]$.
 By one of Kato's numerical range theorems \cite{Kato}, 
 $$\mathcal{W}(f_{s} (A)) \subseteqq  \,\, \overline{co}\{f_{s} (z): \Re(z)\geqq 0 \},$$
 which is clearly a subset of  $\overline{co}\{f(z): \Re(z)\geqq 0 \}$. 
Since $f(A+ s I) = f_{s} (A)$, 
we derive, by using Theorem\,\ref{Th:3-1} and 
Lemma\ref{Le:4-1}, 
$$\mathcal{W}(f (A))\subseteqq\, \overline{co}\{f(z): \Re(z)\geqq 0 \} \subseteqq \{z: \Re(z) \geqq 0\}.$$
\end{proof} 

\begin{remark}\rm
\cite{Kato} consists of many important theorems, however 
a counter example for the Remark in it was given in Remark of \cite{U-}. 
\end{remark}\par  

\begin{corollary}\label{Cor:4-1}\rm 
(i).  $A\in \mathcal{S}_{1} \Longrightarrow  A^r \in \mathcal{S}_{r}\,\, (0<r<1)$.\\ 
(ii).  $A\in \mathcal{S}_{r} \Longrightarrow  f(A)\in \mathcal{S}_{r}.$
\end{corollary}
\begin{proof} 
(i) follows from Proposition\,\ref{Prop:4-2}.\\
Apply the equality 
$$\arg ( (\frac{1}{\lambda}I  - (\lambda  I + A)^{-1}) {\bf x}, {\bf x}) =
\arg (A {\bf y}, \lambda {\bf y} + A {\bf y}),$$
where $\lambda >0$ and ${\bf y} = (\lambda I + A)^{-1}{\bf x}$,
 to Theorem\,\ref{Th:3-1}, and get (ii).
\end{proof}
\begin{proposition}\label{Prop:4-1} Let $n\geqq 2$.\\
{\rm (i)}. Let $A$ be in $\mathcal{S}_{r}$ with $0<r<2/n$ and $B$ in $\mathcal{S}_{2/n}$. If $AB=BA$ and $A^n =B^n$, then $A=B$.\\
{\rm (ii)}. An accretive operator $A$ has the unique  
n-th root $A^{1/n}$ in $\mathcal{S}_{1/n}$. (This fact is well-known; for example, 
 refer to P.178 of \cite{N-F})
\end{proposition}
\begin{proof}
Because of Lemma\,\ref{Lem:3-1}, $\mathcal{N}(B)= \mathcal{N}(B^n) 
=\mathcal{N}(A^n) = \mathcal{N}(A)$, which reduces $A$ and $B$. 
Hence, we may assume $A$ and $B$ are both injective. 
Since $A$ and $B$ are commutative, 
$$0= A^n - B^n = (A -\omega ^{n-1} B)\cdots (A - \omega B) 
 (A - B),$$
where $\omega = e^{\frac{2\pi}{n} i}$. 
Assume 
 $(A - \omega^{n-1} B) \mathbf{y} =0$. Then  
$(A \mathbf{y}, \mathbf{y}) = \omega^{n-1} (B \mathbf{y}, \mathbf{y}) $, and hence   
each side vanishes.  $(A \mathbf{y}, \mathbf{y}) =0$ yields $A \mathbf{y} = 0$. Indeed, 
assume  $\mathbf{v}\!:=A \mathbf{y} \ne 0$, and let  $A'$ be 
the compression of $A$ to the space spanned by 
$\mathbf{y}$ and $\mathbf{v}$. Then  $\mathcal{W}(A')$ is an ellipse 
or a segment with an end point $0$ (see \cite{D-1}); but the latter case deduces $A'$ is normal and $(A' \mathbf{y},  \mathbf{y}) =0$, which results in  
$A' \mathbf{y} = 0$; this contradicts to the assumption. On the other hand, 
the former case contradicts to $A' \in \mathcal{S}_{r}$.
In the end $A \mathbf{y} = {\bf 0}$. Thus
$\mathbf{y}= 0$ arises. By repeating 
this procedure we finally get $A = B$.\\
(ii). By Corollary\ref{Cor:4-1}, $A^{1/n} \in \mathcal{S}_{1/n}\subseteq \mathcal{S}_{2/n}$. 
Assume $B \in \mathcal{S}_{1/n}$ and $A=B^n$. Then by (i) 
$B= A^{1/n}$.
\end{proof}


\section{Operator Means}
The harmonic 
mean $A \, ! \, B$ is defined by $(A^{-1} \nabla B^{-1} )^{-1}$ if $A$, $B$ and $A^{-1} \nabla B^{-1}$ 
are all invertible. We note that the invertibility of $A$ and $B$ does not imply 
that of $A+B$. 
Suppose $A + B$ is invertible. Then for $\epsilon >0$ 
$ (A + \epsilon) \, ! \, (B + \epsilon)$ converges to $2 A (A + B)^{-1} B = 2B (A + B)^{-1} A$ as 
$\epsilon \to 0$ in the norm. We extend the definition of the harmonic mean 
as follows:
\begin{definition}\label{Def:5-1}
Let $A$ and $B$ be accretive operators such that $A + B$ is invertible. Then we define 
the harmonic mean of them by 
$$A \, ! \, B : = 2 A (A + B)^{-1} B.$$
\end{definition}
It is apparent that $A \, ! \, B =B \, ! \, A$,  $(A \, ! \, B)^* =  A^*  \, ! \, B^* $
 and that $A \, ! \, B$ is accretive too. 
If  $A + A^*$ is invertible, then 
$$ A \nabla A^* - A \, ! \, A^* =  \frac{1}{2} ( A - A^* )(A + A^* )^{-1}(A-A^*),  $$
which has appeared first in \cite{U-2019} (cf. \cite{U-2023, U-2024}). 
This formula implies that if  $\Re{A}$ is positive definite,  
\begin{equation}\label{eq:9}
0\leqq A \nabla A^* \leqq  A \, !\, A^*.
\end{equation}
At first sight this seems the reverse 
of the inequality of positive operators; however if $A$ is positive, both sides reduce  
to $A$. 
  Since 
$A\,! \,A^* = (\Re A^{-1})^{-1}$, 
we realize that \eqref{eq:9} is equivalent to  
$$\Re{A} \leqq  (\Re A^{-1})^{-1},$$
which has been shown in \cite{Mathias}.


\begin{proposition}\label{Prop:5-1}
Let $A$ and $B$ be strictly accretive operators. Then 
\begin{itemize}
\item[(i)]
$0\leqq (A+A^*)\, !\, (B+B^*) \leqq (A+B)\, !\, (A^* + B^*) \leqq A \,!\, A^* + B\, !\, B^*$ \\
\item[(ii)]
$0\leqq (A+A^*)\, !\, (B+B^*) \leqq (A \,!\, B) +  (A^* \,! \, B^*) \leqq A \,!\, A^* + B\, !\, B^*$.
 \end{itemize}
\end{proposition}
(The third inequality of (i) and the second one of (ii) are essentially due to \cite{Mathias} and 
\cite{Lin}, respectively, also see \cite{RMF}.)
\begin{proof}
(i).  Since $A + A^* \geqq 0$ and $B + B^* \geqq 0$, 
we have  
 $$(A+A^*)\, !\, (B+B^*) \leqq  (A+A^*)\,\nabla \, (B+B^*)=  (A+B)\nabla\, (A^* + B^*)
 \leqq (A+B)\, !\, (A^* + B^*),$$
where the last inequality follows from \eqref{eq:9}. 
This shows the second inequality. 
 Since $A\,!\, A^* = A^* (A\nabla A^*)^{-1} A$, we have  
$$\begin{pmatrix}A\nabla A^* & A\\ A^* & A\, !\,  A^* \end{pmatrix} \geqq 0, \quad 
\begin{pmatrix}B\nabla B^* & B\\ B^* & B\, !\,  B^* \end{pmatrix} \geqq 0,$$
and hence 
$$ \begin{pmatrix}(A + B)\nabla (A^* +B^*)  & A + B \\ A^*+ B^*  & A\, !\,  A^* + 
B \, !\,  B^* 
\end{pmatrix} \geqq 0,$$ 
 which entails the third inequality.\\
(ii)
Replace $A$ to $A^{-1}$ and $B$ to $B^{-1}$ in (i) and take the inverse of 
each side to get (ii). 
\end{proof} 

In \cite{Drury},  Drury has defined the geometric mean $A\# B$ of matrices $A$ and 
$B$ with positive definite real parts and proved that
$$A\# B = A^{1/2} (A^{-1/2} B A^{-1/2})^{1/2} A^{1/2}$$
and that the real part of $A\# B$ remains positive definite: furthermore, he showed  
$$H A^{-1} H = B$$
for $H$ with positive definite real part if and only if $H= A\# B$. 
Subsequently, general mean $A\sigma_f B$ was defined in \cite{BKS1}, where 
$\sigma_f$ is an operator mean introduced by Kubo-Ando \cite{K-A},  
(see \cite{Lin, RMF} about related topics).\par 
We here extend this definition to accretive operators without assuming 
their real parts are positive definite. To do it we need to add 
the condition $f(1)=1$, namely we assume 
\begin{center}
\it $f(t)>0$ is an operator monotone function given by \eqref{eq:1} and $f(1)=1$.
\end{center}
\rm
Let $A$ and $B$ be accretive and invetible operators. Then 
$A^{-1/2}$ is well-defined as mentioned before.
\begin{definition}\label{Def:5-2}
If 
\begin{equation}\label{eq:10}
{\it Sp}(BA^{-1}) \cap (-\infty, 0] = \emptyset, 
\end{equation}
then 
\begin{equation}\label{eq:11}
A\sigma_f B:= A^{1/2} f( A^{-1/2} B A^{-1/2}) A^{1/2}.
\end{equation}
\end{definition}
From Proposition\,\ref{Prop:2-2} and Example\,\ref{Ex:2-1}
it follows that 
\begin{equation}\label{eq:12}
\begin{split}
&A\sigma_f B = f(0) A  + b B + \int_0^{\infty} 
\frac{1}{\lambda } A ( \lambda A + B)^{-1} B d\mu(\lambda).\\
&A\#_r B =A^{1/2} ( A^{-1/2} B A^{-1/2})^r A^{1/2}= 
 \frac{\sin r \pi}{\pi} \int_0^{\infty} 
 \frac{1}{\lambda}A (\lambda A + B)^{-1}B\lambda ^r d\lambda,
\end{split}
\end{equation}
where $\#_r$ stands for the mean associated with $t^r\,(0<r<1)$ and  
$\#: = \#_{1/2} $ is called the geometric mean. \par
Let $A, B$ be strictly accretive operators. 
Then, since $B+ \lambda A$ is invertible for $\lambda >0$, the condition 
\eqref{eq:10}
is fulfilled. So, $A\sigma_f B$ is defined. \par
Let us furnish an interesting example. 
\begin{example}
$$ A:=\begin{pmatrix} 0 & i \\ i & 1\end{pmatrix},\quad 
B:=\begin{pmatrix} 1 & - i \\ - i & 0 \end{pmatrix}. $$
These are invertible and accretive, but neither of them have 
positive definite real parts. By \eqref{eq:12}
\begin{align*}
A \# B &= \frac{1}{\pi} \int_0^{\infty} 
A ( \lambda A + B)^{-1} B \lambda ^{-1/2} d\lambda \\
 &= \frac{1}{\pi}\int_0^{\infty}
\frac{1}{\sqrt{\lambda}(\lambda^2 - \lambda + 1)}
\begin{pmatrix}\lambda & -i(\lambda -1)\\ -i(\lambda -1)& 1 \end{pmatrix} 
d\lambda \\
&=\frac{1}{\pi} \begin{pmatrix} \pi & 0 \\ 0& \pi \end{pmatrix} =  
\begin{pmatrix} 1 & 0 \\ 0& 1 \end{pmatrix}
\end{align*}
\end{example}

\begin{lemma}\label{Lem:5-0} Let $A$ and $B$ fulfills the 
condition \eqref{eq:10}. Then 
\begin{itemize}
\item[(i)]
 $A\sigma_f B$ is accretive.
\item[(ii)] $(A\#_r B)^{-1}= A^{-1} \#_r B^{-1}$ if $A$ and $B$ are invertible.
\item[(iii)]   For invertible $C$, 
$C^* (A\sigma_f B) C = (C^*A C)\sigma_f (C^* B C)$.
\item[(iv)]
$(A\sigma_f B)^* = (A^* \sigma_f B^*)$.
\end{itemize}
\end{lemma}
\begin{proof}
(i). $2A(\lambda A + B)^{-1} B=\frac{1}{\lambda} (\lambda A \,!\, B)$ evidently shows (i). It is easy to see the rest by using \eqref{eq:11} and \eqref{eq:12}.
\end{proof}
 Since $f(0)\geqq 0$, by \eqref{eq:1}, 
$h(t): = t f(1/t)$ is operator monotone on $(0,  \infty)$ as well. 
\begin{lemma}\label{Le:5-1}
{\rm (i)}.\, $A\sigma_h B = B\sigma_f A$; in particular,
$A\#_r B = B\#_{1-r} A$.\\
{\rm (ii)}.\, If $f(t)= t f(1/t)$, then $A\sigma_f B = B\sigma_f A$
\end{lemma}
\begin{proof}
\begin{align*}
A\sigma_h B & = A^{1/2}  A^{-1/2} B A^{-1/2} f( A^{1/2} B^{-1} A^{1/2}) A^{1/2}
= B A^{-1/2} f( A^{1/2} B^{-1} A^{1/2}) A^{1/2}\\
&= B f(B^{-1} A)= B^{1/2}B^{1/2}
f(B^{-1}A ) B^{-1/2} B^{1/2}= B^{1/2} f(B^{-1/2 }A B^{-1/2}) B^{1/2} \\
&=B\sigma_f A.
\end{align*}
The rest is evident. 
\end{proof}
Recall that $\sigma _f$ is called a {\it symmetric mean}  if 
$f(t)= t f(1/t)$ and $f(1)=1$: then $ f(0) = b$ arises too.

\begin{proposition}\label{Prop:5-3} 
Let $A$ and $B$ be accretive and invertible operators satisfying
${\it Sp}(BA^{-1}) \cap (-\infty, 0] = \emptyset$. Then\\
{\rm (i)}. $A\# B \in \mathcal{S}_{1}$ and 
\begin{center} $(A\# B) A^{-1} (A \# B) =B.$\end{center}
Conversely, if  $X \in \mathcal{S}_{r} \,(0<r<1)$  satisfies  $X A^{-1} X = B$,
 then $X = A\# B$.\\
{\rm (ii)}.\begin{center}
$A(A\# B)^{-1} B= B(A\# B)^{-1} A = A\# B.$
\end{center}
Conversely, if  $X \in \mathcal{S}_{r} \,(0<r<1)$  satisfies
 $A X^{-1} B= X$, then $X= A\# B$. 

\end{proposition}
\begin{proof}
(i). \, $A\# B = A^{1/2}(A^{-1/2} B A^{-1/2})^{1/2} A^{1/2}$
 obviously satisfies the 
required equation. If $X A^{-1} X = B$, then 
$$(A^{-1/2} X A^{-1/2})^2 = A^{-1/2} B A^{-1/2}.$$
 Since
$$(A^{-1/2} B A^{-1/2})^{1/2} = A^{-1/2} (A \# B) A^{-1/2},$$
$A^{-1/2} (A \# B) A^{-1/2}$ commutes to 
$A^{-1/2} X A^{-1/2}$. Hence 
\begin{align*}
&0 =(A^{-1/2} (A \# B) A^{-1/2})^2 - (A^{-1/2} X A^{-1/2})^2 \\
=& \{(A^{-1/2} (A \# B) A^{-1/2}) + (A^{-1/2} X A^{-1/2}) \} 
\{(A^{-1/2} (A \# B) A^{-1/2}) - (A^{-1/2} X A^{-1/2})\}.
\end{align*}
Assume $\{(A^{-1/2} (A \# B) A^{-1/2}) + (A^{-1/2} X A^{-1/2})\} {\bf x} =0$. 
Then $ (A \# B) {\bf y} + X {\bf y} = 0$ for ${\bf y}= A^{-1/2} {\bf x}$. 
We therefore obtain $\Re {((A \# B) {\bf y}, {\bf y})} = \Re{ (X {\bf y}, {\bf y})} =0$, 
which causes $((A \# B) {\bf y}, {\bf y}) =  (X {\bf y}, {\bf y}) =0$,
Since $X \in \mathcal{S}_{r}$, 
by the same reason as the proof of Proposition\,\ref{Prop:4-1}, we obtain 
$X {\bf y} =0$, which implies 
$$(A^{-1/2} B A^{-1/2})^{1/2}{\bf x} = (A^{-1/2} (A \# B) A^{-1/2}){\bf x} =0,$$
and ${\bf x}=0$, because $A$ and $B$ are invertible. We arrive at $A\# B = X$.\\
(ii). Since
$$XA^{-1}X=B \Longleftrightarrow A X ^{-1} B= X,$$
(ii) follows to (i). 
\end{proof}
 
From (ii) of Proposition\,\ref{Prop:5-3} 
it follows that
\begin{align*}
(A\# B)^{-1}& = A^{-1} (A\# B) B^{-1} = A^{-1}\frac{1}{\pi} \int_0^{\infty}\frac{1}{\sqrt{\lambda}}
A (\lambda A + B)^{-1} B d \lambda B^{-1} \\
&= \frac{2}{\pi} \int_0^{\infty} (tA + t^{-1 } B)^{-1} t^{-1} d t.
\end{align*}
We consequently got Brury's definition \eqref{eq:1-0} from our definition. 
\begin{theorem}\label{Th:5-5}
Let $A $ and $B$ be normal and accretive operators 
such that $AB=BA$ and  
${\it Sp}(BA^{-1}) \cap (-\infty, 0] = \emptyset$. Then 
$$A\# B = A^{1/2} B^{1/2}.$$
\end{theorem}
\begin{proof}  $A^{-1/2}$ and $B^{1/2}$ are both normal and in 
$\mathcal{S}_{1/2}$, and they are commutative. Considering the Cartesian 
decomposition       
 $A^{-1/2} = C_1 + i D_1$ and $B^{1/2} = C_2 + i D_2$, $C_i, \,D_i \,(i=1,2)$ are commutative each other. Since 
$|(D_1 {\bf x}, {\bf x})|\leqq (C_1{\bf x} , {\bf x})$ for every ${\bf x}$, 
we get $|D_1|\leqq C_1$,
 and $|D_2|\leqq C_2$ as well.  Consequently $|D_1  D_2|\leqq C_1 C_2$ arises. 
This implies $A^{-1/2} B^{1/2} \in \mathcal{S}_{1}$. From the hypothesis it follows that 
$${\it Sp}((BA^{-1})^{1/2}) = ( {\it Sp}(BA^{-1}) )^{1/2} \subseteqq \{z: |\arg z| < \pi /2\}.$$ 
By Stone's theorem \cite{Stone} the closure of the numerical range of $(BA^{-1})^{1/2}$ is the closed 
convex hull of ${\it Sp}((BA^{-1})^{1/2})$, there is $0<r<1$ so that
 $(BA^{-1})^{1/2}  \in \mathcal{S}_{r}$. Since 
$(A^{-1/2} B^{1/2})^2 = B A^{-1} =( (BA^{-1})^{1/2})^2$, by Proposition\,\ref{Prop:4-1}
 $A^{-1/2} B^{1/2} = (BA^{-1})^{1/2}$. We finally get 
$A\# B = A^{1/2}(A^{-1/2} B A^{-1/2})^{1/2} A^{1/2} = A^{1/2} B^{1/2}$. 
\end{proof}

\section{self and adjoint means}
Let us consider a symmetric mean $\sigma_f$ which is defined by $f(t)$ satisfying 
\begin{center}
$f(t)= t f(1/t), \,\,f(0)=b \geqq 0, \,\, f(1)=1$.
\end{center}  
\begin{theorem}\label{Th:6-1}
Let $A$ be an accretive operator with positive definite real part and 
$\sigma _f$ a symmetric mean. Then 
$A \sigma _f  A^*$ is positive definite. More precisely 
\begin{equation*}
0\leqq A\nabla A^* \leqq A \sigma _f  A^* \leqq A \,!\, A^*.
\end{equation*}
In particular, 
$$
0\leqq A\nabla A^* \leqq A \# A^* \leqq A \,!\, A^*.$$
\end{theorem}
\begin{proof} 
$f(t)$ being given by \eqref{eq:1}, by \eqref{eq:12}, we get
$$ A\sigma_f A^* = f(0) (A + A^*) + \int_0^{\infty} 
\frac{1}{2\lambda^2}  (\lambda A) \,!\, A^* \, d \mu (\lambda).$$
Since $A\sigma_f A^* = A^* \sigma_f A$,  in virtue of (ii) of Proposition\,\ref{Prop:5-1},
\begin{gather*}
\int_0^{\infty}\frac{1}{2\lambda^2}  (\lambda A) \,!\, A^* \, d \mu (\lambda) =
\int_0^{\infty}\frac{1}{4\lambda^2}  \{(\lambda A) \,!\, A^* +  (\lambda A^*) \,!\, A \}\, d \mu (\lambda)\\
\geqq \int_0^{\infty}\frac{1}{4\lambda^2}  \{(\lambda (A + A^*)) \,!\, (A + A^*) =
\int_0^{\infty}\frac{1}{2\lambda^2}  \frac{\lambda}{\lambda + 1} (A + A^*) \,d \mu(\lambda).
\end{gather*}
We therefore arrive at 
$$ A\sigma_f A^* \geqq \{2 f(0) + \int_0^{\infty} 
\frac{1}{\lambda  (\lambda +1)}d \mu (\lambda) \}\frac{A + A^*}{2}= f(1) 
A\nabla A^*.$$
 Since $A\nabla A^* \leqq A\,!\, A^*$, to prove  the second inequality 
it is enough to verify 
$$\Re \frac{1}{2\lambda ^2} ((\lambda A) \,!\, A^*) \leqq \frac{A\,!\, A^*}
{\lambda (\lambda +1)}.$$ This is equivalent to 
\begin{align*}
&\frac{1}{2\lambda}\left( (\lambda A + A^*)^{-1} + (A + \lambda A^*)^{-1}\right)
\leqq \frac{2}{\lambda (\lambda + 1)}(A + A^*)^{-1}\\
\Longleftrightarrow & \frac{A + A^*}{2} \leqq (\frac{\lambda}{\lambda + 1} A + 
\frac{1}{\lambda + 1}A^*) \,!\,  (\frac{1}{\lambda + 1} A + 
\frac{\lambda}{\lambda + 1}A^*).
\end{align*}
By taking use of \eqref{eq:9} again, we obtain the last inequality. 
\end{proof}
For an operator monotone function $f(t)>0$ on $(0, \infty)$ with $f(1) = 1$, 
$h(t):= \frac{1}{2}\{ f( t) + t f(1/t)\}$ defines a symmetric mean. Hence, we have 
\begin{align}
&A\nabla A^* \leqq  \frac{1}{2} \{ A \sigma _f  A^* + A^* \sigma _f  A \}
 \leqq A \,!\, A^*,\notag\\
&A\nabla A^* \leqq  \frac{1}{2} \{ A \#_r  A^* + A \#_{1-r} A^* \}
 \leqq A \,!\, A^*.\label{eq:14}
\end{align}
\begin{proposition}\label{Prop:6-2}
Let $A$ be strictly accretive. Then \\
(i). $A^* \# A = |A|$ if and only if $A$ is normal.\\
(ii). $A\#_r A^* = A ^{1-r} A^{*r}$ if $A$ is normal. 
\end{proposition}
\begin{proof} (i). 
If $A$ is normal, then $|A| A^{-1} |A|= A^*$. Since $|A|$ is positive definite, 
by Proposition\,\ref{Prop:5-3}, 
we get $A^* \# A = |A|$. Suppose conversely $A^* \# A = |A|$. Then 
$|A| A^{-1} |A| = A^*$. We consequently get
$$A A^*= |A| (A^*)^{-1}|A||A| A^{-1} |A| = A^* A.$$
(ii). $A^{-1}  A^*$ is normal and 
${\it Sp}(A^{-1} A^{*}) \cap (-\infty, 0] = \emptyset$. Putting 
$A=\int_{{\it Sp}(A)} z d E_z$, 
by the similar way as (ii) of Remark\,\ref{Re:2-2} we get 
$$(A^{-1}  A^*)^r = \int_{{\it Sp}(A)} \left(\frac{\overline{z}}{z}\right)^r d E_z = A^{-r} A^{*r},$$
which indicate (ii).
\end{proof}
\begin{corollary}\label{Cor:6-3}
Let $A$ be strictly accretive and normal. Then 
\begin{equation}\label{eq:15}
|A|=  \frac{1}{\pi}\int_0^{\infty}A (\lambda A + A^*)^{-1} A^* 
\lambda^{-1/2} d \lambda 
\end{equation}
 \end{corollary}
This evidently follows from \eqref{eq:12} and the previous proposition, and 
induces the simplest case:
\begin{example}
If $-\pi/2 < \theta < \pi/2$, then
$$1 = \frac{1}{\pi} \int_0^{\infty} \frac{1}{(\lambda e^{i\theta}  +  e^{-i\theta})
\sqrt{\lambda} } d \lambda.$$
\end{example}
One can confirm this elementary formula by calculating the residue.   

 We note that the positive definiteness of $\Re{A}$ 
is not assumed in the following.
 \begin{corollary}\label{Cor:6-4}
Let $A$ be a normal and accretive operator.
Then so is $A^{1-r} A^{*r}$ for $0\leqq r \leqq1$, and  
$$A + A^* \leqq   A^{1-r} A^{*r} + A^r A^{*(1-r)} .$$
\end{corollary}
\begin{proof} Let $\epsilon>0$ and substitute  $A+\epsilon I $ into the first 
inequality of \eqref{eq:14}. Because of Proposition\,\ref{Prop:6-2},  we get 
$$(A + \epsilon I) + (A + \epsilon I)^* \leqq \{ (A+ \epsilon I)^{1-r}
 (A + \epsilon I)^{*r} + (A + \epsilon I)^r (A+ \epsilon I)^{*(1-r)} \}.$$
By Lemma\,\ref{Le:2} $(A + \epsilon I)^r$ converges to $A^r$ as $\epsilon \to 0$, and so on. 
We consequently obtain the required inequality. The rest is clear.
\end{proof}
 \begin{remark}\rm
In the above corollary, the case $r=1/2$ means a simple inequality 
$\Re(A) \leqq |A|$, and the cases $r=0$ or $r=1$ are trivial; so it interpolates them. 
When we rewrite the left side as $((1-r)A + r A^*) + (r A + (1-r) A^*)$, the inequality looks 
like a reverse of the Young inequality for positive operators; however the both sides are equal 
if $A=A^*$.  
%
\end{remark}

\begin{proposition}\label{Prop:6-5} Let $A$ be an accretive operator with positive definite real part and 
$\sigma _f$ a symmetric mean.
Let $h(t)> 0$ be an operator monotone function, not necessarily symmetric,
 on $(0, \infty)$. Then 
$$h(A\nabla A^*) \leqq h(A)\nabla h(A^*)\leqq h(A) \sigma _f  h(A^*)
 \leqq h(A) \,!\, h(A^*) \leqq h(A \,!\, A^*).$$
\end{proposition}
 \begin{proof}
By \eqref{eq:9}, for $\lambda > 0$
$$\left(\frac{(\lambda I + A) + (\lambda I + A^*)}{2}\right) ^{-1} \geqq 
\frac{(\lambda I + A)^{-1} + (\lambda I + A^*)^{-1}}{2}.$$
By applying \eqref{eq:6} to $h(t)$, we get the first inequality. This also shows that the real part of 
 $h(A)$ is positive definite too, which implies $h(A) \sigma _f  h(A^*)$ is well-defined.
 To see the last inequality, 
substitute $A^{-1}$ to $A$  and $\frac{1}{h(1/t)}$ to $h(t)$ in the first one.
Two other inequalities come from the previous theorem.
\end{proof}
The first inequality was shown in \cite{BKS1} for matrices in a different way. 


\end{document}